\documentclass[oneside,english]{amsart}
\usepackage[T1]{fontenc}
\usepackage[latin9]{inputenc}
\usepackage{geometry}
\usepackage{array}
\usepackage{float}
\usepackage{multirow}
\usepackage{amsthm}
\usepackage{amstext}
\usepackage{amssymb}
\usepackage{graphicx}
\usepackage{esint}

\makeatletter

\providecommand{\tabularnewline}{\\}

\numberwithin{equation}{section}
\numberwithin{figure}{section}
\theoremstyle{plain}
\newtheorem{thm}{\protect\theoremname}
  \theoremstyle{definition}
  \newtheorem{defn}[thm]{\protect\definitionname}
  \theoremstyle{plain}
  \newtheorem{prop}[thm]{\protect\propositionname}
  \theoremstyle{remark}
  \newtheorem{rem}[thm]{\protect\remarkname}
  \theoremstyle{plain}
  \newtheorem{lem}[thm]{\protect\lemmaname}
  \theoremstyle{plain}
  \newtheorem{conjecture}[thm]{\protect\conjecturename}

\usepackage{shuffle}

\makeatother

  \providecommand{\conjecturename}{Conjecture}
  \providecommand{\definitionname}{Definition}
  \providecommand{\lemmaname}{Lemma}
  \providecommand{\propositionname}{Proposition}
  \providecommand{\remarkname}{Remark}
\providecommand{\theoremname}{Theorem}

\begin{document}
\address[Minoru Hirose]{Multiple Zeta Research Center, Kyushu University}
\email{m-hirose@math.kyushu-u.ac.jp}
\address[Nobuo Sato]{National Center for Theoretical Sciences, National Taiwan University}
\email{saton@ncts.ntu.edu.tw}

\title[Iterated integrals on $\mathbb{P}^{1}\setminus\{0,1,\infty,z\}$ and
a class of relations of MZVs ]{Iterated integrals on $\mathbb{P}^{1}\setminus\{0,1,\infty,z\}$
and a class of relations among multiple zeta values}

\author{Minoru Hirose and Nobuo Sato}
\begin{abstract}
In this paper we consider iterated integrals on $\mathbb{P}^{1}\setminus\{0,1,\infty,z\}$
and define a class of $\mathbb{Q}$-linear relations among them, which
arises from the differential structure of the iterated integrals with
respect to $z$. We then define a new class of $\mathbb{Q}$-linear
relations among the multiple zeta values by taking their limits of
$z\rightarrow1$, which we call \emph{confluence relations} (i.e.,
the relations obtained by the confluence of two punctured points).
One of the significance of the confluence relations is that it gives
a rich family and seems to exhaust all the linear relations among
the multiple zeta values. As a good reason for this, we show that
confluence relations imply both the regularized double shuffle relations
and the duality relations.
\end{abstract}


\maketitle

\section{Introduction}

The multiple zeta values, MZVs in short, are the real numbers defined
by the multiple Dirichlet series 
\[
\zeta(k_{1},\ldots,k_{d})=\sum_{0<m_{1}<\cdots<m_{d}}\frac{1}{m_{1}^{k_{1}}\cdots m_{d}^{k_{d}}}
\]
or equivalent iterated integral
\[
\zeta(k_{1},\dots,k_{d})=(-1)^{d}I(0;1,\{0\}^{k_{1}-1},\dots,1,\{0\}^{k_{d}-1};1)
\]
where
\[
I(0;a_{1},\dots,a_{n};1)=\int_{0<t_{1}<\cdots<t_{n}<1}\prod_{i=1}^{n}\frac{dt_{i}}{t_{i}-a_{i}}.
\]
Here $(k_{1},\ldots,k_{d})$ is called the index, $k:=k_{1}+\cdots+k_{d}$
is called the weight and $d$ is called the depth of the MZV. We assume
that $k_{1},\ldots,k_{d-1}\geq1$ and $k_{d}\geq2$ for convergence
of the series/integral. It is known that there are many $\mathbb{Q}$-linear
relations among the multiple zeta values, and their structure is one
of the main interest of the study of multiple zeta values.

There are three known classes of linear relations among the MZVs which
are conjectured to exhaust all the linear relations among them:
\begin{itemize}
\item Associator relations
\item Regularized double shuffle relations (RDS)
\item Kawashima's relations \cite{Kawashima} (when the product is expanded
by the shuffle relation)
\end{itemize}
All three families above come from different backgrounds. The associator
relations come from geometric relations of Drinfel'd associator of
KZ-equation. The regularized double shuffle relations come from the
series and iterated integral expression of multiple zeta values. Kawashima's
relations come from the Newton series which interpolate the multiple
harmonic sums.

The first purpose of this paper is to define a new class of linear
relations among MZVs which we call \emph{confluence relations} based
on the theory of iterated integrals on $\mathbb{P}^{1}\setminus\{0,1,\infty,z\}$,
and the second purpose is to show that this class includes the regularized
double shuffle relations (and the duality relations). Clearly, the
latter result implies that ``confluence relations'' are also expected
to exhaust all linear relations of MZVs. Together with the known inclusion
between Kawashima relations, Associator relations, RDS plus duality
relations, the implication relations of the classes can be summarized
as follows.

\begin{center}
\includegraphics[scale=0.8]{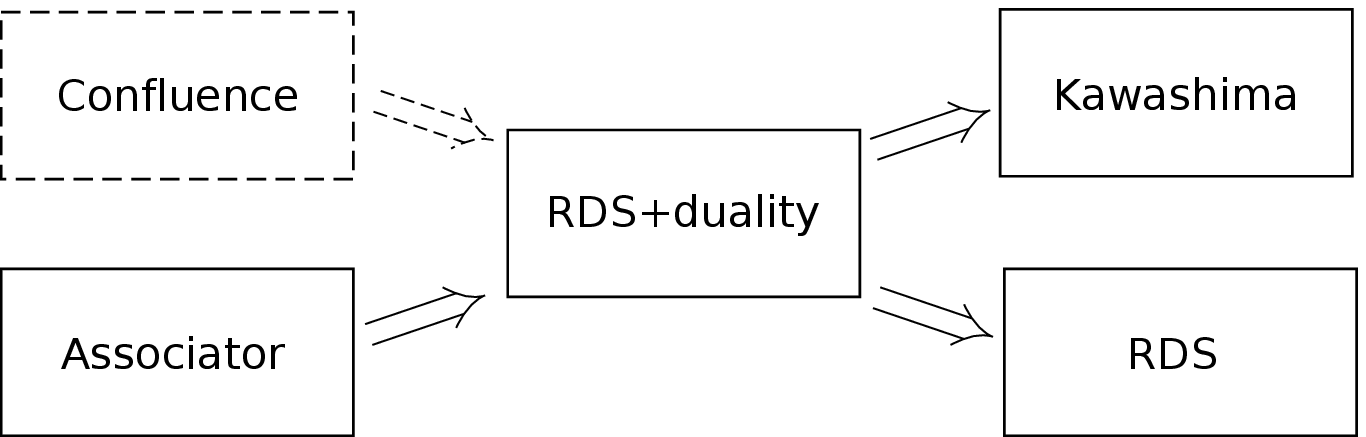}
\par\end{center}

Here, the implication $({\rm Associator})\Rightarrow({\rm RDS})$
was proved by Furusho \cite{Furusho_Associator_DSR}, and the implication
$({\rm RDS}+{\rm duality})\Rightarrow({\rm Kawashima})$ was proved
by Kaneko-Yamamoto \cite{Kaneko_Yamamoto_integralseries}.

The sketch of our method is as follows. To introduce the confluence
relation, we first consider the holomorphic function of $z\in\mathbb{C}\setminus[0,1]$
defined by the iterated integral
\[
I(0;a_{1},\dots,a_{n};1)\ \ \ \ (a_{1},\dots,a_{n}\in\{0,1,z\}).
\]
Then the derivatives of these functions are again expressible in terms
of the functions of the same form. The \emph{standard relations} is
a class of linear relations of such holomorphic functions which arise
from this differential structure. Since the limit $z\to1$ of any
relation among these functions should give a relation among MZVs,
we define the ``confluence relation'' by the limit $z\to1$ of the
``standard relations'', which is the basic idea of the confluence
relation. However, as the limit does not always exist, we need to
invent some tricks/techniques to deal with the behavior of divergence
and to obtain a suitable meaningful substitute of the limit. Next
we prove that the confluence relations includes RDS and the duality.
The key to these proofs are the differential formulas of the shuffle
and the stuffle products and the duality (Theorem \ref{thm:Three_applications},
\cite{HS_AlgebraicDiff}).

In Section \ref{sec:Standard-relations} we introduce some algebraic
settings and define the standard relations of iterated integrals (Definition
\ref{def:ISD}, Theorem \ref{thm:ISD_is_relation}). Then in Section
\ref{sec:Confluence-relations} we formulate the confluence relations
of MZVs (Definition \ref{Def:CF}, Theorem \ref{thm:main_relation}).
In Section \ref{sec:Regularized-double-shuffle} we show that the
confluence relations includes RDS (Theorem \ref{thm:RDS-from-ST})
and duality relations (Theorem \ref{thm:duality-from-ST}). At the
end, in Appendix \ref{sec:proofs-of-SomeComplementary} we give proofs
of complementary propositions and in Appendix \ref{sec:Table-of-relation}
we give a table of the confluence relations up to weight 4.

\subsection{Notations.}

We denote by $(A,\circ)$ an algebra $A$ equipped with the bilinear
product $\circ$. For $\mathbb{Z}$-modules $A$ and $B$, we denote
by $\mathrm{Hom}(A,B)$ the $\mathbb{Z}$-module formed by $\mathbb{Z}$-linear
maps from $A$ to $B$. Also, we denote $\mathrm{End}(A):=\mathrm{Hom}(A,A)$.

\section{\label{sec:Standard-relations}Standard relations of iterated integrals}

\subsection{Algebraic settings and differential formulas for iterated integrals}

Let $\mathcal{A}_{z}:=\mathbb{Z}\left\langle e_{0},e_{1},e_{z}\right\rangle $
(resp. $\mathcal{A}:=\mathbb{Z}\left\langle e_{0},e_{1}\right\rangle \subset\mathcal{A}_{z}$)
be a non-commutative ring generated by formal symbols $e_{0},e_{1},e_{z}$
(resp. $e_{0},e_{1}$). We define the subspaces $\mathcal{A}_{z}^{0}\subset\mathcal{A}_{z}^{1}\subset\mathcal{A}_{z}$
and $\mathcal{A}^{0}\subset\mathcal{A}^{1}\subset\mathcal{A}$ by
\begin{align*}
\mathcal{A}_{z}^{0} & :=\mathbb{Z}\oplus\mathbb{Z}e_{z}\oplus\bigoplus_{\substack{a\in\{1,z\}\\
b\in\{0,z\}
}
}e_{a}\mathcal{A}_{z}e_{b},\\
\mathcal{A}^{0} & :=\mathbb{Z}\oplus e_{1}\mathcal{A}e_{0}\quad\left(=\mathcal{A}\cap\mathcal{A}_{z}^{0}\right),\\
\mathcal{A}_{z}^{1} & :=\mathbb{Z}\oplus\bigoplus_{a\in\{1,z\}}e_{a}\mathcal{A}_{z},\\
\mathcal{A}^{1} & :=\mathbb{Z}\oplus e_{1}\mathcal{A}\quad\left(=\mathcal{A}\cap\mathcal{A}_{z}^{1}\right).
\end{align*}

For $z\in\mathbb{C}\setminus[0,1]$, we define a $\mathbb{Z}$-linear
map $L:\mathcal{A}_{z}^{0}\rightarrow\mathbb{C}$ by 

\[
L(e_{a_{1}}\cdots e_{a_{m}}):=\int_{0<t_{1}<\cdots<t_{m}<1}\prod_{i=1}^{m}\frac{dt_{i}}{t_{i}-a_{i}}.
\]
Note that, for $w\in\mathcal{A}_{z}^{0}$, $L(w)$ can be regarded
as a holomorphic function of $z$ on $\mathbb{C}\setminus[0,1]$.
In particular, the multiple zeta value is expressed as 
\[
\zeta(k_{1},\ldots,k_{d})=L((-e_{1})e_{0}^{k_{1}-1}\cdots(-e_{1})e_{0}^{k_{d}-1})
\]
by $L$. 
\begin{defn}
For $\alpha,\beta\in\{0,1,z\}$, we define a linear operator $\partial_{\alpha,\beta}$
on $\mathcal{A}_{z}$ by
\[
\partial_{\alpha,\beta}\left(e_{a_{1}}\cdots e_{a_{n}}\right):=\sum_{i=1}^{n}\left(\delta_{\{a_{i},a_{i+1}\},\{\alpha,\beta\}}-\delta_{\{a_{i-1},a_{i}\},\{\alpha,\beta\}}\right)e_{a_{1}}\cdots\widehat{e_{a_{i}}}\cdots e_{a_{n}}
\]
where $a_{0}=0,\, a_{n+1}=1$ and $\delta_{S,T}$ denotes the Kronecker
delta i.e.,
\[
\delta_{S,T}=\begin{cases}
1 & S=T\\
0 & S\neq T,
\end{cases}
\]
for sets $S$ and $T$.
\end{defn}
It is easy to check that $\partial_{\alpha,\beta}(\mathcal{A}_{z}^{0})\subset\mathcal{A}_{z}^{0}$.
Note that $\partial_{z,0}+\partial_{z,1}+\partial_{1,0}=0$ on $\mathcal{A}_{z}^{0}$
(see Proposition \ref{prop:deriv_sum_z0_z1_10}). 

The operator $\partial_{z,a}$ is related to the actual differentiation
$d/dz$ as follows.
\begin{prop}[{\cite[Theorem 2.1]{HIST}}]
\textup{\label{prop:differential formula} $\frac{d}{dz}L\left(w\right)=\sum_{a\in\{0,1\}}\frac{1}{z-a}L\left(\partial_{z,a}\left(w\right)\right).$}
\end{prop}

\subsection{Shuffle product, stuffle product, duality, and their algebraic differential
formulas}
\begin{defn}
The shuffle product $\shuffle:\mathcal{A}_{z}\times\mathcal{A}_{z}\to\mathcal{A}_{z}$
is a bilinear map defined by $w\shuffle1=1\shuffle w=w$ for $w\in\mathcal{A}_{z}$
and
\[
e_{a}u\shuffle e_{b}v=e_{a}(u\shuffle e_{b}v)+e_{b}(e_{a}u\shuffle v)\qquad(u,v\in\mathcal{A}_{z}).
\]

\end{defn}
Note that, for $?\in\{\emptyset,0,1\}$, $\mathcal{A}_{z}^{?}$ and
$\mathcal{A}^{?}$ become commutative rings by the shuffle product.

\begin{defn}[\cite{HS_AlgebraicDiff}]
\label{Stuffle product}The (generalized) stuffle product $*:\mathcal{A}\times\mathcal{A}_{z}\to\mathcal{A}_{z}$
is a bilinear map defined by $u*1=u,\,1*v=v$ for $u\in\mathcal{A}$,
$v\in\mathcal{A}_{z}$ and
\begin{align*}
e_{a}u*e_{b}v & =e_{ab}(u*e_{b}v+e_{a}u*v-e_{0}(u*v))\,\qquad(u\in\mathcal{A},v\in\mathcal{A}_{z}).
\end{align*}

\end{defn}
Note that, for $?\in\{\emptyset,0,1\},$ $\mathcal{A}^{?}$ becomes
a commutative ring by the stuffle product, and $\mathcal{A}_{z}^{?}$
becomes an $(\mathcal{A}^{?},*)$ module by the stuffle product (it
follows from the associativity of more generalized stuffle products
shown in \cite[Proposition 6]{HS_AlgebraicDiff}). See \cite{HS_AlgebraicDiff}
for the compatibility with the standard definition of the stuffle
product.

\begin{defn}
The duality map $\tau_{z}:\mathcal{A}_{z}\to\mathcal{A}_{z}$ is an
anti-automorphism (i.e., $\tau_{z}(uv)=\tau_{z}(v)\tau_{z}(u)$) defined
by $\tau_{z}(e_{0})=e_{z}-e_{1}$, $\tau_{z}(e_{1})=e_{z}-e_{0}$
and $\tau_{z}(e_{z})=e_{z}$.
\end{defn}
Note that $\tau_{z}(\mathcal{A}_{z}^{0})\subset\mathcal{A}_{z}^{0}$
from the definition.
\begin{defn}
For $w\in\mathcal{A}_{z}^{1}$, the shuffle regularization ${\rm reg}_{\shuffle}(w)\in\mathcal{A}_{z}^{0}$
is defined by ${\rm reg}_{\shuffle}(w):=w_{0}$ where $w_{0}$ is
defined by the unique expression
\[
w=\sum_{i=0}^{{\rm deg}(w)}w_{i}\shuffle e_{1}^{i}\ \ \ (w_{i}\in\mathcal{A}_{z}^{0}).
\]

\end{defn}
The following proposition is a fundamental property of the shuffle
and stuffle products and the duality map.
\begin{prop}
\label{prop:ShStDualIdentity}We have:
\begin{enumerate}
\item For $u,v\in\mathcal{A}_{z}^{0}$, 
\[
L(u\shuffle v)=L(u)L(v).
\]

\item For $u\in\mathcal{A}_{z}^{0}$ and $v\in\mathcal{A}_{z}^{0}$,
\[
L(u*v)=L(u)L(v).
\]

\item For $u\in\mathcal{A}_{z}^{0}$
\[
L(\tau_{z}(u))=L(u).
\]

\end{enumerate}
\end{prop}
\begin{proof}
(1) and (2 is just a special case of the shuffle and stuffle product
identities of hyperlogarithms \cite{BBBL_stuffle}. For (3), see \cite[Theorem 1.1]{HIST}.
\end{proof}
The following theorem gives fundamental relations between the operators
$\shuffle$, $*$, $\tau_{z}$ and $\partial_{z,c}$.
\begin{thm}
\cite[Theorem 10]{HS_AlgebraicDiff}\label{thm:Three_applications}Let
$c\in\{0,1\}$.
\begin{enumerate}
\item For $u,v\in\mathcal{A}_{z}$,
\[
\partial_{z,c}(u\shuffle v)=(\partial_{z,c}u)\shuffle v+u\shuffle(\partial_{z,c}v).
\]

\item For $u\in\mathcal{A}^{1}$ and $v\in\mathcal{A}_{z}$, 
\[
\partial_{z,c}(u*v)=u*(\partial_{z,c}v).
\]

\item For $u\in\mathcal{A}_{z}^{0}$, 
\[
\tau_{z}^{-1}\circ\partial_{z,c}\circ\tau_{z}(u)=\partial_{z,c}u.
\]

\end{enumerate}
\end{thm}
\begin{rem}
It is also possible to prove Proposition \ref{prop:ShStDualIdentity}
from Theorem \ref{thm:Three_applications}. 
\end{rem}

\subsection{Standard relations}

We define $\mathrm{Const}\in\mathrm{Hom}(\mathcal{A}_{z},\,\mathcal{A})$
by 
\[
\mathrm{Const}(e_{a_{1}}\cdots e_{a_{n}}):=\begin{cases}
e_{a_{1}}\cdots e_{a_{n}} & \mbox{ if }a_{i}\in\{0,1\}\ \text{for all }i,\\
0 & \mbox{ if }a_{i}=z\ \text{for some }i.
\end{cases}
\]
Note that $\lim_{z\to\infty}L(w)=L({\rm Const}(w))$ for $w\in\mathcal{A}_{z}^{0}$
since $\lim_{z\to0}L(w')=0$ for $w'\in\mathcal{A}_{z}^{0}\cap\mathcal{A}_{z}e_{z}\mathcal{A}_{z}$.
Note that $\mathrm{Const}(u\shuffle v)={\rm Const}(u)\shuffle{\rm Const}(v)$
for $u,v\in\mathcal{A}_{z}$ and ${\rm Const}(u*v)=u*{\rm Const}(v)$
for $u\in\mathcal{A}$, $v\in\mathcal{A}_{z}$ (for the latter identity,
see Proposition \ref{prop:ConstStuffle}).

We define $\varphi_{\otimes}\in{\rm Hom}(\mathcal{A}_{z}^{0},\,\mathcal{A}^{0}\otimes\mathbb{Z}\left\langle e_{0},e_{z}\right\rangle )$
by
\[
\varphi_{\otimes}(w):=\sum_{\substack{r\in\mathbb{Z}_{\geq0}\\
b_{1},\ldots,b_{r}\in\{0,z\}
}
}\mathrm{Const}\left(\partial_{1,b_{1}}\cdots\partial_{1,b_{r}}w\right)\otimes e_{b_{1}}\cdots e_{b_{r}}.
\]
Then we can show that $\varphi_{\otimes}(\mathcal{A}_{z}^{0})\subset\mathcal{A}^{0}\otimes\left(\mathbb{Z}\left\langle e_{0},e_{z}\right\rangle \cap\mathcal{A}_{z}^{0}\right)$
(see Proposition \ref{prop:phi_tensor_in_good_space}). We define
$\varphi_{\shuffle},\varphi_{*}\in{\rm End}(\mathcal{A}_{z}^{0})$
by composing $\varphi_{\otimes}$ and shuffle or stuffle product,
i.e. 
\begin{align*}
\varphi_{\shuffle}(w) & :=\sum_{\substack{r\in\mathbb{Z}_{\geq0}\\
b_{1},\ldots,b_{r}\in\{0,z\}
}
}\mathrm{Const}\left(\partial_{1,b_{1}}\cdots\partial_{1,b_{r}}w\right)\shuffle e_{b_{1}}\cdots e_{b_{r}}\\
\varphi_{*}(w) & :=\sum_{\substack{r\in\mathbb{Z}_{\geq0}\\
b_{1},\ldots,b_{r}\in\{0,z\}
}
}\mathrm{Const}\left(\partial_{1,b_{1}}\cdots\partial_{1,b_{r}}w\right)*e_{b_{1}}\cdots e_{b_{r}}.
\end{align*}

\begin{rem}
The map $\varphi_{\shuffle}$ is a ring endomorphism of $(\mathcal{A}_{z}^{0},\shuffle)$,
and the map $\varphi_{*}$ is a homomorphism as $(\mathcal{A}^{0},*)$-module
(see Proposition \ref{prop:phi_shuffle_hom} and \ref{prop:phi_stuffle_hom}).\end{rem}
\begin{lem}
\label{lem:idem}We have $\varphi_{\otimes}\circ\varphi_{\shuffle}=\varphi_{\otimes}\circ\varphi_{*}=\varphi_{\otimes}$.\end{lem}
\begin{proof}
The claim follows from (1) and (2) of Theorem \ref{thm:Three_applications}
since $\partial_{1,b}(e_{b_{1}}\cdots e_{b_{r}})=\delta_{b,b_{r}}e_{b_{1}}\cdots e_{b_{r-1}}$
for $b,b_{1},\dots,b_{r}\in\{0,z\}$.\end{proof}
\begin{prop}
\label{prop:ISD_equiv}The following six subspaces of $\mathcal{A}_{z}^{0}$
are equal.
\begin{enumerate}
\item The image of $({\rm id}-\varphi_{\shuffle}):\mathcal{A}_{z}^{0}\to\mathcal{A}_{z}^{0}$.
\item The kernel of $\varphi_{\shuffle}:\mathcal{A}_{z}^{0}\to\mathcal{A}_{z}^{0}$.
\item The image of $({\rm id}-\varphi_{*}):\mathcal{A}_{z}^{0}\to\mathcal{A}_{z}^{0}$.
\item The kernel of $\varphi_{*}:\mathcal{A}_{z}^{0}\to\mathcal{A}_{z}^{0}$.
\item The kernel of $\varphi_{\otimes}:\mathcal{A}_{z}^{0}\to\mathcal{A}^{0}\otimes\mathbb{Z}\left\langle e_{0},e_{z}\right\rangle $.
\item The set $\{w\in\mathcal{A}_{z}^{0}\mid{\rm Const}(\partial_{z,\alpha_{1}}\cdots\partial_{z,\alpha_{r}}w)=0\ {\rm for}\ r\geq0,\alpha_{1},\dots,\alpha_{r}\in\{0,1\}\}$.
\end{enumerate}
\end{prop}
\begin{proof}
The equality of (5) and (6) follows from $\partial_{1,0}+\partial_{z,0}+\partial_{z,1}=0$.
Since $\varphi_{\shuffle}$ and $\varphi_{*}$ factor through $\varphi_{\otimes}$,
we have $\ker(\varphi_{\otimes})\subset\ker(\varphi_{\shuffle})$
and $\ker(\varphi_{\otimes})\subset\ker(\varphi_{*})$. Since $({\rm id}-\varphi_{\shuffle})(w)=w$
for $w\in\ker(\varphi_{\shuffle})$, we have $\ker(\varphi_{\shuffle})\subset{\rm Im}({\rm id}-\varphi_{\shuffle})$.
Similarly, since $({\rm id}-\varphi_{*})(w)=w$ for $w\in\ker(\varphi_{*})$,
we have $\ker(\varphi_{*})\subset{\rm Im}({\rm id}-\varphi_{*})$.
From Lemma \ref{lem:idem}, we have ${\rm Im}({\rm id}-\varphi_{\shuffle})\subset\ker(\varphi_{\otimes})$
and ${\rm Im}({\rm id}-\varphi_{*})\subset\ker(\varphi_{\otimes})$.
Thus all these subspaces are equal.\end{proof}
\begin{defn}
\label{def:ISD}We define the set of \emph{standard relations} $\mathcal{I}_{{\rm ST}}$
as the subspace of $\mathcal{A}_{z}^{0}$ given in Proposition \ref{prop:ISD_equiv}.\end{defn}
\begin{rem}
$\mathcal{I}_{{\rm SD}}$ forms an ideal of $(\mathcal{A}_{z}^{0},\shuffle)$
since $\mathcal{I}_{{\rm SD}}=\ker(\varphi_{\shuffle})$ and $\varphi_{\shuffle}$
is a ring endomorphism of $(\mathcal{A}_{z}^{0},\shuffle)$. By the
similar reasoning, $\mathcal{I}_{{\rm SD}}$ is also $(\mathcal{A}^{0},*)$-module.
\end{rem}
The next theorem states that ``standard relations'' are in fact relations
of iterated integrals.
\begin{thm}
\label{thm:ISD_is_relation}For $w\in\mathcal{I}_{{\rm ST}}$, $L(w)=0$. \end{thm}
\begin{proof}
Since $\mathcal{I}_{{\rm ST}}=\{w\in\mathcal{A}_{z}^{0}\mid{\rm Const}(\partial_{z,\alpha_{1}}\cdots\partial_{z,\alpha_{r}}w)=0\ {\rm for}\ r\geq0,\alpha_{1},\dots,\alpha_{r}\in\{0,1\}\}$,
we have 
\[
v\in\mathcal{I}_{{\rm ST}}\Leftrightarrow\left({\rm Const}(v)=0,\ \partial_{z,0}(v)\in\mathcal{I}_{{\rm ST}}\ {\rm and}\ \partial_{z,1}(v)\in\mathcal{I}_{{\rm ST}}\right)
\]
for $v\in\mathcal{A}_{z}^{0}$. We prove the theorem by induction
on the degree of $w$. Let $w\in\mathcal{I}_{{\rm ST}}$. Since ${\rm Const}(w)=0$,
we have
\[
\lim_{z\to\infty}L(w)=0.
\]
From the induction hypothesis, $L(\partial_{z,0}w)=L(\partial_{z,1}w)=0$.
Thus
\begin{align*}
\frac{d}{dz}L(w) & =\frac{1}{z}L(\partial_{z,0}w)+\frac{1}{z-1}L(\partial_{z,1}w)=0.
\end{align*}
Thus the claim $L(w)=0$ is proved. 
\end{proof}

\section{\label{sec:Confluence-relations}Confluence relations}

From now on, we consider the limit $z\to1_{+0}$ of the standard relations.
We define the subspaces $\mathcal{A}_{z}^{-2}\subset\mathcal{A}_{z}^{-1}\subset\mathcal{A}_{z}^{0}$
by
\[
\mathcal{A}_{z}^{-2}:=\mathbb{Z}\oplus e_{1}\mathcal{A}_{z}e_{0}\oplus e_{z}\mathcal{A}_{z}e_{0}
\]
and
\[
\mathcal{A}_{z}^{-1}:=\mathcal{A}_{z}^{-2}\mathbb{Z}\left\langle e_{z}\right\rangle .
\]
We define an endomorphism $|_{a\rightarrow b}:\mathcal{A}_{z}\to\mathcal{A}_{z}$
($a,b\in\{0,1,z\}$) by 
\[
e_{x}|_{a\rightarrow b}:=\begin{cases}
e_{b} & x=a\\
e_{x} & x\neq a.
\end{cases}
\]
For $w\in\mathcal{A}_{z}^{-2}$, $\lim_{z\to1}L(w)$ exists and equals
to $L(\left.w\right|_{z\to1})$. For general $w\in\mathcal{A}_{z}^{0}$,
$\lim_{z\to1}L(w)$ does not exist, but the following lemma holds.
\begin{lem}
\label{lem:Pw}For $w\in\mathcal{A}_{z}^{0}$, there exists a natural
number $m$ and a unique polynomial $P_{w}(T)\in\mathbb{R}[T]$ such
that
\[
L(w)-P_{w}(\log(z-1))=O((z-1)\log^{m}(z-1)).
\]
\end{lem}
\begin{proof}
The uniqueness is obvious. By Theorem \ref{thm:ISD_is_relation},
$L(w)=L(\varphi_{\shuffle}(w))$ which is a sum of $L(u)L(v)$ with
$u\in\mathcal{A}^{0}$ and $v\in\mathcal{A}_{z}^{0}\cap\mathbb{Z}\left\langle e_{0},e_{z}\right\rangle $,
and thus the proof of the lemma can be reduced to the case $w\in\mathcal{A}_{z}^{0}\cap\mathbb{Z}\left\langle e_{0},e_{z}\right\rangle $
. Put $w=e_{z}e_{0}^{k_{1}-1}\cdots e_{z}e_{0}^{k_{d}-1}$. Then we
have
\[
L(w)=(-1)^{d}{\rm Li}_{k_{1},\dots,k_{d}}(z^{-1})
\]
where ${\rm Li}_{k_{1},\dots,k_{d}}$ is the multiple polylogarithm.
Thus, the lemma follows from the following well-known fact (c.f. \cite[p311]{IKZ})
that there exists $m\in\mathbb{Z}_{>0}$ and $Q\in\mathbb{R}[T]$
such that 
\[
{\rm Li}_{k_{1},\dots,k_{d}}(t)=Q(\log(1-t))+O((1-t)\log^{m}(1-t)).
\]

\end{proof}
Thus, for general $w\in\mathcal{A}_{z}^{0}$, we consider $P_{w}(0)$
instead of $\lim_{z\to1}L(w)$. We define $\lambda':\mathcal{A}_{z}^{-2}\to\mathcal{A}^{0}$
by
\[
\lambda'(w)=\left.w\right|_{z\to1}.
\]
Then obviously $P_{w}(0)=L(\lambda'(w))$ for $w\in\mathcal{A}_{z}^{-2}$.
Now we construct a extension $\lambda:\mathcal{A}_{z}^{0}\to\mathcal{A}^{0}$
of $\lambda'$ such that $P_{w}(0)=L(\lambda(w))$ for $w\in\mathcal{A}_{z}^{0}$.

Note that there exists an isomorphism
\[
\mathcal{A}_{z}^{-2}\otimes(\mathcal{A}_{z}^{0}\cap\mathbb{Z}\left\langle e_{1},e_{z}\right\rangle )\simeq\mathcal{A}_{z}^{0}\ ;\ u\otimes v\mapsto u\shuffle v
\]
(see Proposition \ref{prop:bijectivity_f} for bijectivity). We denote
the inverse of this isomorphism by ${\rm reg}_{\{z,1\}}$.
\begin{defn}
We define a homomorphism $N:\mathcal{A}_{z}^{0}\to\mathcal{A}_{z}^{-1}$
by the composition
\[
\mathcal{A}_{z}^{0}\xrightarrow[{\rm reg}_{\{z,1\}}]{\simeq}\mathcal{A}_{z}^{-2}\otimes\left(\mathcal{A}_{z}^{0}\cap\mathbb{Z}\left\langle e_{1},e_{z}\right\rangle \right)\xrightarrow[{\rm id}\otimes f]{\simeq}\mathcal{A}_{z}^{-2}\otimes\left(\mathcal{A}_{z}^{0}\cap\mathbb{Z}\left\langle e_{0},e_{z}\right\rangle \right)\xrightarrow[u\otimes v\mapsto u\shuffle v]{\twoheadrightarrow}\mathcal{A}_{z}^{-1}
\]
where 
\[
f:=\tau_{z}\mid_{\mathcal{A}_{z}^{0}\cap\mathbb{Z}\left\langle e_{1},e_{z}\right\rangle }\ \ (=\varphi_{\shuffle}\mid_{\mathcal{A}_{z}^{0}\cap\mathbb{Z}\left\langle e_{1},e_{z}\right\rangle }=\varphi_{*}\mid_{\mathcal{A}_{z}^{0}\cap\mathbb{Z}\left\langle e_{1},e_{z}\right\rangle }).
\]

\end{defn}
From the definition, we have $L(N(w))=L(w)$. Note that there exists
an isomorphism
\[
\mathcal{A}_{z}^{-2}\otimes\mathbb{Z}\left\langle e_{z}\right\rangle \simeq\mathcal{A}_{z}^{-1}\ ;\ u\otimes v\mapsto u\shuffle v.
\]
We denote the inverse of this isomorphism by ${\rm reg}_{\{z\}}$.
\begin{defn}
We define a homomorphism $\lambda:\mathcal{A}_{z}^{0}\to\mathcal{A}^{0}$
by the composition
\[
\mathcal{A}_{z}^{0}\xrightarrow{N}\mathcal{A}_{z}^{-1}\xrightarrow{{\rm reg}_{\{z\}}}\mathcal{A}_{z}^{-2}\otimes\mathbb{Z}\left\langle e_{z}\right\rangle \xrightarrow{{\rm id}\otimes{\rm const}}\mathcal{A}_{z}^{-2}\otimes\mathbb{Z}=\mathcal{A}_{z}^{-2}\xrightarrow{\lambda'}\mathcal{A}^{0}.
\]
\end{defn}
\begin{prop}
\label{prop:Pw0_eq_Llambdaw}We have $P_{w}(0)=L(\lambda(w))$ for
$w\in\mathcal{A}_{z}^{0}$.\end{prop}
\begin{proof}
Put ${\rm reg}_{\{z\}}^{'}=({\rm id}\otimes{\rm const})\circ{\rm reg}_{\{z\}}$.
Since $L(N(w))=L(w)$, it is enough to prove that $P_{w}(0)=L(\lambda'({\rm reg}_{\{z\}}^{'}w))$
for $w\in\mathcal{A}_{z}^{-1}$. Define $w_{0},w_{1},\dots\in\mathcal{A}_{z}^{-2}$
by
\[
{\rm reg}_{\{z\}}w=\sum_{k\geq0}w_{k}\otimes e_{z}^{k}.
\]
Since $P_{u\shuffle e_{z}^{k}}(T)=P_{u}(T)\frac{T^{k}}{k!}$ for $u\in\mathcal{A}_{z}^{0}$
from Lemma \ref{lem:Pw},
\[
P_{w}(T)=\sum_{k\geq0}L(\lambda'(w_{k}))\frac{T^{k}}{k!}.
\]
Thus 
\[
P_{w}(0)=L(\lambda'(w_{0}))=L(\lambda'({\rm reg}_{\{z\}}^{'}w)).
\]
\end{proof}
\begin{thm}
\label{thm:main_relation}For $w\in\mathcal{I}_{{\rm ST}}$, we have
\[
L(\lambda(w))=0.
\]
\end{thm}
\begin{proof}
It follows from Theorem \ref{thm:ISD_is_relation} and Proposition
\ref{prop:Pw0_eq_Llambdaw}.\end{proof}
\begin{defn}
\label{Def:CF}We define a $\mathbb{Z}$-module $\mathcal{I}_{{\rm CF}}$
by 
\[
\mathcal{I}_{{\rm CF}}=\{\lambda(w)\mid w\in\mathcal{I}_{{\rm ST}}\}\subset\mathcal{A}^{0}
\]
and say $w\in\ker\left(L:\mathcal{A}^{0}\rightarrow\mathbb{R}\right)$
is a \emph{confluence relation} if $w\in\mathcal{I}_{{\rm CF}}$.\end{defn}
\begin{rem}
Since $\mathcal{I}_{{\rm ST}}={\rm Im}({\rm id}-\varphi_{\shuffle})$,
$\mathcal{I}_{{\rm CF}}=\{\lambda(w-\varphi_{\shuffle}(w))\mid w\in\mathcal{A}_{z}^{0}\}$.
Since there are $4\cdot3^{k-2}$ words of length $k$ for $k\geq2$,
we can obtain $4\cdot3^{k-2}$ relations of multiple zeta values of
weight $k$ for $k\geq2$. For example, let $w=e_{z}e_{1}e_{0}$.
Then we have
\[
w-\varphi_{\shuffle}(w)=e_{z}e_{1}e_{0}-\left(e_{1}e_{0}\shuffle e_{z}+1\shuffle e_{z}e_{0}^{2}-1\shuffle e_{z}e_{0}e_{z}\right).
\]
Thus 
\[
{\rm reg}_{\{1,z\}}(w-\varphi_{\shuffle}(w))=(e_{z}e_{1}e_{0}-e_{z}e_{0}^{2}-2e_{z}^{2}e_{0})\otimes1+(-e_{1}e_{0}+e_{z}e_{0})\otimes e_{z}.
\]
Since $\tau_{z}(1)=1$ and $\tau_{z}(e_{z})=e_{z}$,
\[
N(w-\varphi_{\shuffle}(w))=(e_{z}e_{1}e_{0}-e_{z}e_{0}^{2}-2e_{z}^{2}e_{0})\shuffle1+(-e_{1}e_{0}+e_{z}e_{0})\shuffle e_{z}.
\]
Thus 
\[
\lambda(w-\varphi_{\shuffle}(w))=\left.(e_{z}e_{1}e_{0}-e_{z}e_{0}^{2}-2e_{z}^{2}e_{0})\right|_{z\to1}=-e_{1}^{2}e_{0}-e_{1}e_{0}^{2}.
\]
Therefore $-e_{1}e_{0}^{2}-e_{1}e_{0}^{2}\in\mathcal{I}_{{\rm CF}}\subset\ker(L)$.
This gives $L(-e_{1}^{2}e_{0}-e_{1}e_{0}^{2})=-\zeta(1,2)+\zeta(3)=0$.
\end{rem}

\begin{rem}
Since $\lambda(N(w))=\lambda(w)$ and $\varphi_{\shuffle}(N(w))=\varphi_{\shuffle}(w)$,
we have $\mathcal{I}_{{\rm CF}}=\{\lambda(w-\varphi_{\shuffle}(w))\mid w\in\mathcal{A}_{z}^{-1}\}$.
Since $\lambda={\rm reg}\circ\mid_{z\to1}$ on $\mathcal{A}_{z}^{-1}$
and $\varphi_{\shuffle}(\mathcal{A}_{z}^{0})\subset\mathcal{A}_{z}^{-1}$,
we have 
\[
\mathcal{I}_{{\rm CF}}=\{{\rm reg}((w-\varphi_{\shuffle}(w))\mid_{z\to1})\mid w\in\mathcal{A}_{z}^{-1}\}.
\]
\end{rem}
\begin{conjecture}
\label{Conj: Confluence-relations}The confluence relations exhaust
all the relations of the multiple zeta values, i.e.,
\[
\mathcal{I}_{{\rm CF}}\otimes\mathbb{Q}={\rm ker}(L:\mathcal{A}^{0}\to\mathbb{R})\otimes\mathbb{Q}.
\]

\end{conjecture}

\section{\label{sec:Regularized-double-shuffle}Regularized double shuffle
and duality relations are in $\mathcal{I}_{{\rm CF}}$.}

In this section we shall prove that the regularized double shuffle
relations and the duality relations of the multiple zeta values are
confluence relations. The results in this section give theoretical
supports of our Conjecture \ref{Conj: Confluence-relations}.

\subsection{The confluence relations imply the regularized double shuffle relations}

We denote by $\mathcal{I}_{{\rm RDS}}$ the ideal of $(\mathcal{A}^{0},\shuffle)$
generated by the regularized double shuffle relations, i.e.,
\[
\{{\rm reg}_{\shuffle}(u\shuffle v-u*v)\mid u\in\mathcal{A}^{1},v\in\mathcal{A}^{0}\}.
\]

\begin{lem}
\label{lem:partial_vs_reg}For $c\in\{0,1\}$, 
\[
\partial_{z,c}{\rm reg}_{\shuffle}(w)={\rm reg}_{\shuffle}(\partial_{z,c}w).
\]
\end{lem}
\begin{proof}
From $\partial_{z,c}(e_{1}^{i})=0$ and (1) of Theorem \ref{thm:Three_applications},
we have
\[
\partial_{z,c}(u\shuffle e_{1}^{i})=\partial_{z,c}(u)\shuffle e_{1}^{i}\ \ \ (c\in\{0,1\}).
\]
Thus the lemma is proved. \end{proof}
\begin{thm}
\label{thm:RDS-from-ST}$\mathcal{I}_{{\rm RDS}}\subset\mathcal{I}_{{\rm CF}}$.\end{thm}
\begin{proof}
Fix $u\in\mathcal{A}^{1}$ and $v\in\mathcal{A}^{0}$. Since $\mathcal{I}_{{\rm RDS}}$
is an ideal, it suffices to prove that
\[
{\rm reg}_{\shuffle}(u\shuffle v-u*v)\in\mathcal{I}_{{\rm CF}}.
\]
Put 
\[
f(w):={\rm reg}_{\shuffle}(u\shuffle w-u*w)\ \ \ (w\in\mathcal{A}_{z}^{0}\cap\mathbb{Z}\left\langle e_{0},e_{z}\right\rangle ).
\]
Then we have $\partial_{z,c}f(w)=f(\partial_{z,c}w)$ for all $c\in\{0,1\}$
by Lemma \ref{lem:partial_vs_reg} and (1), (2) of Theorem \ref{thm:Three_applications}.
Therefore, we can show that $f(w)\in\mathcal{I}_{{\rm ST}}$ by the
induction on the degree of $w$ since ${\rm Const}(f(w))=0$. Especially,
we have
\[
f(v_{z})={\rm reg}_{\shuffle}(u\shuffle v_{z}-u*v_{z})\in\mathcal{I}_{{\rm ST}}
\]
where $v_{z}:=v|_{1\rightarrow z}\in\mathbb{Z}\oplus e_{z}\mathbb{Q}\left\langle e_{0},e_{z}\right\rangle e_{0}$.
Since $f(v_{z})\in\mathcal{A}_{z}^{-2}$, we have
\[
\mathcal{I}_{{\rm CF}}\ni\lambda(f(v_{z}))=\left.f(v_{z})\right|_{z\to1}={\rm reg}_{\shuffle}(u\shuffle v-u*v).
\]

\end{proof}

\subsection{The confluence relations imply the duality relations}

Let $\tau_{\infty}$ be an antiautomorphism of $\mathcal{A}$ defined
by $\tau_{\infty}(e_{0})=-e_{1}$, $\tau_{\infty}(e_{1})=-e_{0}$.
Set $\Delta(w):=w-\tau_{\infty}(w)$ and $\Delta_{z}(w):=w-\tau_{z}(w)$.
Let $\mathcal{I}_{{\rm \Delta}}$ denote the ideal of $\mathcal{A}^{0}$
generated $\{\Delta(w)\mid w\in\mathcal{A}^{0}\}.$ In this section
we prove $\mathcal{I}_{\Delta}\subset\mathcal{I}_{{\rm ST}}$.
\begin{lem}
\label{cor:V-supportedness_of_delta-z}For $w\in\mathcal{A}_{z}^{0}$,
$\varphi_{\otimes}(\Delta_{z}(w))\in\mathcal{I}_{{\rm \Delta}}\otimes\mathcal{I}_{{\rm ST}}$.\end{lem}
\begin{proof}
For $r\in\mathbb{Z}_{\geq0}$ and $a_{1},\ldots,a_{r}\in\{0,1\}$,
\[
\mathrm{Const}\left(\partial_{z,a_{1}}\cdots\partial_{z,a_{r}}\Delta_{z}(w)\right)=\mathrm{Const}\left(\Delta_{z}(\partial_{z,a_{1}}\cdots\partial_{z,a_{r}}w)\right)
\]
for $w\in\mathcal{A}_{z}^{0}$ by (3) of Theorem \ref{thm:Three_applications}.
Since $\mathrm{Const}(\Delta_{z}(w))=\Delta(\mathrm{Const}(w))\in I_{{\rm \Delta}}$
for $w\in\mathcal{A}_{z}$, this proves the corollary.\end{proof}
\begin{thm}
\label{thm:duality-from-ST}$\mathcal{I}_{\Delta}\subset\mathcal{I}_{{\rm ST}}$.\end{thm}
\begin{proof}
Put $\mathcal{I}_{\Delta}^{(k)}:=\{\Delta(u)\mid u\in\mathcal{A}^{0},\ \deg u\leq k\}$.
We prove the claim $\mathcal{I}_{\Delta}^{(k)}\in\mathcal{I}_{{\rm ST}}$
by the induction on $k$. Take $u\in\mathcal{A}^{0}$ such that $\deg u\leq k$.
Since $u\in\mathcal{A}^{0}$, we can assume that $u=e_{1}u'e_{0}$
where $u'\in\mathcal{A}$. Put $w=\Delta(u)+\Delta((e_{z}-e_{1})u'e_{0})\in\mathcal{A}_{z}^{-2}$.
Then we have ${\rm Const}(w)=0$ and $\varphi_{\otimes}(w)\in\mathcal{I}_{\Delta}^{(k-1)}\otimes\mathcal{A}$
from Lemma \ref{cor:V-supportedness_of_delta-z}. Thus ${\rm reg}_{\shuffle}\left(\left.\varphi_{\shuffle}(w)\right|_{z\to1}\right)\in\mathcal{I}_{\Delta}^{(k)}\shuffle\mathcal{A}^{0}\subset\mathcal{I}_{{\rm ST}}$
from the induction assumption. On the other hand, we have ${\rm reg}_{\shuffle}\left(\left.w\right|_{z\to1}\right)=\Delta(u)$.
Thus the theorem is proved since $\Delta(u)={\rm reg}_{\shuffle}(\left.(w-\varphi(w))\right|_{z\to1})+{\rm reg}_{\shuffle}(\left.\varphi(w)\right|_{z\to1})\in\mathcal{I}_{{\rm ST}}$. 
\end{proof}

\section*{Acknowledgements}

This work was supported by a Postdoctoral fellowship at the National
Center for Theoretical Sciences.

The authors would like to thank Erik Panzer for some useful comments
on a draft of this paper.

\bibliographystyle{plain}
\bibliography{reference}

\def\Dbar{\leavevmode\lower.6ex\hbox to 0pt{\hskip-.23ex \accent"16\hss}D}
  \def\cfac#1{\ifmmode\setbox7\hbox{$\accent"5E#1$}\else
  \setbox7\hbox{\accent"5E#1}\penalty 10000\relax\fi\raise 1\ht7
  \hbox{\lower1.15ex\hbox to 1\wd7{\hss\accent"13\hss}}\penalty 10000
  \hskip-1\wd7\penalty 10000\box7}
  \def\cftil#1{\ifmmode\setbox7\hbox{$\accent"5E#1$}\else
  \setbox7\hbox{\accent"5E#1}\penalty 10000\relax\fi\raise 1\ht7
  \hbox{\lower1.15ex\hbox to 1\wd7{\hss\accent"7E\hss}}\penalty 10000
  \hskip-1\wd7\penalty 10000\box7}
\begin{thebibliography}{1}

\bibitem{BBBL_stuffle}
Jonathan~M. Borwein, David~M. Bradley, David~J. Broadhurst, and Petr
  Lison\v{e}k.
\newblock Special values of multiple polylogarithms.
\newblock {\em Trans. Amer. Math. Soc.}, 353(3):907--941, 2001.

\bibitem{Furusho_Associator_DSR}
Hidekazu Furusho.
\newblock Double shuffle relation for associators.
\newblock {\em Ann. of Math. (2)}, 174(1):341--360, 2011.

\bibitem{HIST}
Minoru Hirose, Kohei Iwaki, Nobuo Sato, and Koji Tasaka.
\newblock Sum/duality formulas for iterated integrals and multiple zeta values.
\newblock {\em preprint}, 2017.
\newblock {\tt arXiv:1704.06387v1 [math.NT]}.

\bibitem{HS_AlgebraicDiff}
Minoru Hirose and Nobuo Sato.
\newblock Algebraic differential formulas for the shuffle, stuffle and duality
  relations of iterated integrals.
\newblock {\em preprint}, 2018.
\newblock {\tt arXiv:1801.03165v1 [math.NT]}.

\bibitem{IKZ}
Kentaro Ihara, Masanobu Kaneko, and Don Zagier.
\newblock Derivation and double shuffle relations for multiple zeta values.
\newblock {\em Compos. Math.}, 142(2):307--338, 2006.

\bibitem{Kaneko_Yamamoto_integralseries}
Masanobu Kaneko and Shuji Yamamoto.
\newblock A new integral-series identity of multiple zeta values and
  regularizations.
\newblock {\em preprint}, 2016.
\newblock {\tt arXiv:1605.03117v3 [math.NT]}.

\bibitem{Kawashima}
Gaku Kawashima.
\newblock A class of relations among multiple zeta values.
\newblock {\em J. Number Theory}, 129(4):755--788, 2009.

\end{thebibliography}

\appendix

\section{\label{sec:proofs-of-SomeComplementary}proofs of some complementary
propositions}
\begin{lem}
\label{lem:deriv_a_a}For $\alpha\in\{0,1,z\}$ and $u\in\mathcal{A}_{z}^{0}$,
$\partial_{\alpha,\alpha}(u)=0$.\end{lem}
\begin{proof}
Let $u=e_{a_{1}}\cdots e_{a_{n}}$. Put $(a_{0},a_{n+1})=(0,1)$,
$S=\{1\leq i\leq n\mid a_{i}=a_{i-1}=\alpha\}$ and $T=\{1\leq i\leq n\mid a_{i}=a_{i+1}=\alpha\}$.
Then $S=\{i+1\mid i\in T\}$ since $a_{1}\neq a_{0}$ and $a_{n}\neq a_{n+1}$.
From the definition of $\partial_{\alpha,\alpha}$, we have
\begin{align*}
\partial_{\alpha,\alpha}u & =\sum_{i\in S}e_{a_{1}}\cdots\widehat{e_{a_{i}}}\cdots e_{a_{n}}-\sum_{i\in T}e_{a_{1}}\cdots\widehat{e_{a_{i}}}\cdots e_{a_{n}}\\
 & =\sum_{i\in S}e_{a_{1}}\cdots\widehat{e_{a_{i}}}\cdots e_{a_{n}}-\sum_{i\in T}e_{a_{1}}\cdots\widehat{e_{a_{i+1}}}\cdots e_{a_{n}}\\
 & =0.
\end{align*}
\end{proof}
\begin{prop}
\label{prop:deriv_sum_z0_z1_10}For $u\in\mathcal{A}_{z}^{0}$, $\partial_{z,0}(u)+\partial_{z,1}(u)+\partial_{1,0}(u)=0$.\end{prop}
\begin{proof}
Put $u=e_{a_{1}}\cdots e_{a_{n}}$ and $X=\{(0,0),(1,1),(z,z),(z,0),(z,1),(1,0)\}$.
From the definition, we have
\begin{align*}
\sum_{(\alpha,\beta)\in X}\partial_{\alpha,\beta}(u) & =\sum_{i=1}^{n}\sum_{(\alpha,\beta)\in X}\left(\delta_{\{a_{i},a_{i+1}\},\{\alpha,\beta\}}-\delta_{\{a_{i-1},a_{i}\},\{\alpha,\beta\}}\right)e_{a_{1}}\cdots\widehat{e_{a_{i}}}\cdots e_{a_{n}}\\
 & =\sum_{i=1}^{n}\left(1-1\right)e_{a_{1}}\cdots\widehat{e_{a_{i}}}\cdots e_{a_{n}}\\
 & =0.
\end{align*}
Thus $\partial_{z,0}(u)+\partial_{z,1}(u)+\partial_{1,0}(u)=0$ by
Lemma \ref{lem:deriv_a_a}.\end{proof}
\begin{prop}
\label{prop:ConstStuffle}For $u\in\mathcal{A}$ and $v\in\mathcal{A}_{z}$,
${\rm Const}(u*v)=u*{\rm Const}(v)$.\end{prop}
\begin{proof}
Put $\mathcal{I}:=\mathcal{A}_{z}e_{z}\mathcal{A}_{z}=\ker({\rm Const}).$
It suffices to show that ${\rm Const}(u*v)\in\mathcal{I}$ for all
monomials $u\in\mathcal{A}$ and $v\in\mathcal{I}$. We prove this
by induction on the sum of degrees of $u$ and $v$. Put $u=e_{a}u'$
and $v=e_{b}v'$. Then from the induction hypothesis, we have $u'*v\in\mathcal{I}$,
and $u*v',u'*v'\in\mathcal{I}$ if $b\neq z$. If $a=0$ then we have
$u*v\in\mathcal{I}$ since $e_{0}u'*v=e_{0}(u'*v)$. Assume that $a=1$.
Then we have
\[
u*v=e_{b}(u'*v+u*v'-e_{0}(u'*v'))\in\mathcal{I}
\]
since either $e_{b}=e_{z}$ or $u'*v+u*v'-e_{0}(u'*v')\in\mathcal{I}$
holds.\end{proof}
\begin{prop}
\label{prop:phi_tensor_in_good_space}$\varphi_{\otimes}(\mathcal{A}_{z}^{0})\subset\mathcal{A}^{0}\otimes\left(\mathbb{Z}\left\langle e_{0},e_{z}\right\rangle \cap\mathcal{A}_{z}^{0}\right)$.\end{prop}
\begin{proof}
From the definition of $\varphi_{\otimes}$, it is enough to prove
that ${\rm Const}(\partial_{1,0}(u))=0$ for $u\in\mathcal{A}_{z}^{0}$.
For a word $w=e_{a_{1}}\cdots e_{a_{n}}$, put 
\[
\deg_{z}(w)=\#\{1\leq i\leq n\mid a_{i}=z\}.
\]
Put 
\[
\mathcal{A}_{z}^{0,k}:=\bigoplus_{\deg_{z}(w)=k}\mathbb{Z}w.
\]
Then $\partial_{1,0}(\mathcal{A}_{z}^{0,k})\subset\mathcal{A}_{z}^{0,k}$.
Therefore ${\rm Const}(\partial_{1,0}(u))=0$ for all $u\in\bigoplus_{k\geq1}\mathcal{A}_{z}^{0,k}$.
Since $\partial_{1,0}=-\partial_{z,0}-\partial_{z,1}$, $\partial_{1,0}(u)=0$
for $u\in\mathcal{A}_{z}^{0,0}$. This proves the claim.\end{proof}
\begin{prop}
\label{prop:phi_shuffle_hom}For $u,v\in\mathcal{A}_{z}^{0}$,
\[
\varphi_{\shuffle}(u\shuffle v)=\varphi_{\shuffle}(u)\shuffle\varphi_{\shuffle}(v).
\]
\end{prop}
\begin{proof}
By (1) of Theorem \ref{thm:Three_applications}, 
\[
\partial_{1,b_{1}}\cdots\partial_{1,b_{r}}(u\shuffle v)=\sum_{k=0}^{r}\sum_{\sigma\in S_{k,r-k}}\partial_{1,b_{\sigma(1)}}\cdots\partial_{1,b_{\sigma(k)}}(u)\shuffle\partial_{1,b_{\sigma(k+1)}}\cdots\partial_{1,b_{\sigma(r)}}(v)
\]
where $S_{k,r-k}:=\left\{ \sigma\in\mathfrak{S}_{r}\left|\substack{\sigma(1)<\cdots<\sigma(k)\\
\sigma(k+1)<\cdots<\sigma(r)
}
\right.\right\} $. Since 
\[
\mathrm{Const}\left(w\shuffle w'\right)=\mathrm{Const}(w)\shuffle\mathrm{Const}(w')
\]
for $w,w'\in\mathcal{A}_{z}$, 
\begin{align*}
\varphi_{\shuffle}(u\shuffle v)= & \sum_{\substack{k,l\in\mathbb{Z}_{\geq0}}
}\sum_{\substack{\sigma\in S_{k,l}\\
b_{1},\ldots,b_{k+l}\in\{0,z\}
}
}\mathrm{Const}(\partial_{b_{\sigma(1)},\ldots,b_{\sigma(k)}}(u))\shuffle\mathrm{Const}(\partial_{b_{\sigma(k+1)},\ldots,b_{\sigma(k+l)}}(v))\\
 & \qquad\qquad\qquad\qquad\qquad\qquad\qquad\qquad\qquad\qquad\qquad\qquad\qquad\shuffle e_{b_{1}}\cdots e_{b_{k+l}}.
\end{align*}
For each $\sigma\in S_{k,l}$, we change the labeling of the suffices
by $b_{\sigma(i)}=a_{i}$ for $1\leq i\leq k+l$. Thus,
\begin{align*}
 & \sum_{\substack{\sigma\in S_{k,l}}
}\sum_{\substack{b_{1},\ldots,b_{k+l}\in\{0,z\}}
}\mathrm{Const}(\partial_{b_{\sigma(1)},\ldots,b_{\sigma(k)}}(u))\shuffle\mathrm{Const}(\partial_{b_{\sigma(k+1)},\ldots,b_{\sigma(k+l)}}(v))\\
 & \qquad\qquad\qquad\qquad\qquad\qquad\qquad\qquad\qquad\qquad\qquad\qquad\qquad\shuffle e_{b_{1}}\cdots e_{b_{k+l}}\\
= & \sum_{\substack{\sigma\in S_{k,l}}
}\sum_{\substack{a_{1},\ldots,a_{k+l}\in\{0,z\}}
}\mathrm{Const}(\partial_{a_{1},\ldots,a_{k}}(u))\shuffle\mathrm{Const}(\partial_{a_{k+1},\ldots,a_{k+l}}(v))\\
 & \qquad\qquad\qquad\qquad\qquad\qquad\qquad\qquad\qquad\qquad\qquad\qquad\qquad\shuffle e_{a_{\sigma^{-1}(1)}}\cdots e_{a_{\sigma^{-1}(k+l)}}.
\end{align*}
Since $\sum_{\substack{\sigma\in S_{k,l}}
}e_{a_{\sigma^{-1}(1)}}\cdots e_{a_{\sigma^{-1}(k+l)}}=e_{a_{1}}\cdots e_{a_{k}}\shuffle e_{a_{k+1}}\cdots e_{a_{k+l}}$, it follows that $\varphi_{\shuffle}(u\shuffle v)=\varphi_{\shuffle}(u)\shuffle\varphi_{\shuffle}(v).$\end{proof}
\begin{prop}
\label{prop:phi_stuffle_hom}For $u\in\mathcal{A}^{0}$ and $v\in\mathcal{A}_{z}^{0}$,
\[
\varphi_{*}(u*v)=u*\varphi_{*}(v).
\]
\end{prop}
\begin{proof}
By (2) of Theorem \ref{thm:Three_applications}, we have
\begin{align*}
\varphi_{*}(u*v)= & \sum_{r\geq0}\sum_{b_{1},\dots,b_{r}\in\{0,z\}}{\rm Const}(u*\partial_{1,b_{1}}\cdots\partial_{1,b_{r}})*e_{b_{1}}\cdots e_{b_{r}}\\
= & u*\sum_{r\geq0}\sum_{b_{1},\dots,b_{r}\in\{0,z\}}{\rm Const}(\partial_{1,b_{1}}\cdots\partial_{1,b_{r}})*e_{b_{1}}\cdots e_{b_{r}}\\
= & u*\varphi_{*}(v).
\end{align*}
\end{proof}
\begin{prop}
\label{prop:bijectivity_f}The homomorphism $f:\mathcal{A}_{z}^{-2}\otimes(\mathcal{A}_{z}^{0}\cap\mathbb{Z}\left\langle e_{1},e_{z}\right\rangle )\simeq\mathcal{A}_{z}^{0}$
defined by
\[
f(u\otimes v)=u\shuffle v
\]
is bijective.\end{prop}
\begin{proof}
Let $X_{k}=\{u\in\mathcal{A}_{z}^{0}\cap\mathbb{Z}\left\langle e_{1},e_{z}\right\rangle \mid\deg u\leq k\}$.
Put $A=\mathcal{A}_{z}^{-2}\otimes(\mathcal{A}_{z}^{0}\cap\mathbb{Z}\left\langle e_{1},e_{z}\right\rangle )$.
Define filtrations $F_{k}$ on $A$ and $\mathcal{A}_{z}^{0}$ by
$F_{k}A=\mathcal{A}_{z}^{-2}\otimes X_{k}$ and $F_{k}\mathcal{A}_{z}^{0}=\mathcal{A}_{z}^{-2}X_{k}$.
Put ${\rm gr}_{k}A=F_{k}A/F_{k-1}(A)$ and ${\rm gr}_{k}\mathcal{A}_{z}^{0}=F_{k}A_{z}^{0}/F_{k-1}\mathcal{A}_{z}^{0}$.
Then the induced map ${\rm gr}_{k}f:{\rm gr}_{k}A\to{\rm gr}_{k}\mathcal{A}_{z}^{0}$
of each graded piece is obviously bijective since the shuffle product
map $u\otimes v\mapsto u\shuffle v$ is equal to just a concatenation
map $u\otimes v\mapsto uv$. 
\end{proof}

\section{A Table of confluence relations\label{sec:Table-of-relation}}

By virtue of the explicit representation of the standard relations
(Proposition \ref{prop:ISD_equiv}, (1)), the confluence relations
can be explicitly given as 
\[
\left\{ \left.\lambda(w-\varphi_{\shuffle}(w))\right|w\in\mathcal{A}_{z}^{0}\right\} .
\]
Thus, we give a table of the confluence relations up to weight 4 as
follows. $ $We omit the case $w\in\mathbb{Z}\left\langle e_{0},e_{1}\right\rangle \cup\mathbb{Z}\left\langle e_{0},e_{z}\right\rangle \cup\mathbb{Z}\left\langle e_{1},e_{z}\right\rangle $
since $\lambda(w-\varphi_{\shuffle}(w))=0$ in this case. For each
$w\in\mathcal{A}_{z}^{0}$, $L(\lambda(w-\varphi_{\shuffle}(w)))=0$
where
\[
L(e_{1}e_{0}^{k_{1}-1}\cdots e_{1}e_{0}^{k_{d}-1})=(-1)^{d}\zeta(k_{1},\dots,k_{d}).
\]
For example, the table below says that we can obtain a relation
\[
-3\zeta(4)+5\zeta(2,2)+13\zeta(1,3)-4\zeta(1,1,2)=0
\]
from the case $w=e_{1}e_{z}e_{1}e_{0}$.

\begin{table}[H]
\begin{tabular}{|c|c|c||c|c|c|}
\hline 
weight & $w$ & $\lambda(w-\varphi_{\shuffle}(w))$ & weight & $w$ & $\lambda(w-\varphi_{\shuffle}(w))$\tabularnewline
\hline 
\hline 
\multirow{3}{*}{3} & $e_{z}e_{1}e_{0}$ & $-e_{1}e_{0}^{2}-e_{1}^{2}e_{0}$ & \multirow{10}{*}{4} & $e_{1}e_{z}^{2}e_{0}$ & $2e_{1}e_{0}e_{1}e_{0}+6e_{1}^{2}e_{0}^{2}+3e_{1}^{3}e_{0}$\tabularnewline
\cline{2-3} \cline{5-6} 
 & $e_{1}e_{z}e_{0}$ & $2e_{1}e_{0}^{2}+2e_{1}^{2}e_{0}$ &  & $e_{1}e_{z}e_{0}e_{z}$ & $-2e_{1}e_{0}e_{1}e_{0}-6e_{1}^{2}e_{0}^{2}-3e_{1}^{3}e_{0}$\tabularnewline
\cline{2-3} \cline{5-6} 
 & $e_{1}e_{0}e_{z}$ & $-e_{1}e_{0}^{2}-e_{1}^{2}e_{0}$ &  & $e_{1}e_{z}e_{0}^{2}$ & $3e_{1}e_{0}^{3}+2e_{1}e_{0}e_{1}e_{0}+6e_{1}^{2}e_{0}^{2}$\tabularnewline
\cline{1-3} \cline{5-6} 
\multirow{7}{*}{4} & $e_{z}^{2}e_{1}e_{0}$ & $e_{1}e_{0}e_{1}e_{0}+e_{1}^{2}e_{0}^{2}+e_{1}^{3}e_{0}$ &  & $e_{1}e_{z}e_{1}e_{0}$ & $3e_{1}e_{0}^{3}+5e_{1}e_{0}e_{1}e_{0}+13e_{1}^{2}e_{0}^{2}+4e_{1}^{3}e_{0}$\tabularnewline
\cline{2-3} \cline{5-6} 
 & $e_{z}e_{0}e_{1}e_{z}$ & $-4e_{1}^{2}e_{0}^{2}-e_{1}^{3}e_{0}$ &  & $e_{1}e_{0}e_{z}^{2}$ & $e_{1}e_{0}e_{1}e_{0}+e_{1}^{2}e_{0}^{2}+e_{1}^{3}e_{0}$\tabularnewline
\cline{2-3} \cline{5-6} 
 & $e_{z}e_{0}e_{1}e_{0}$ & $-e_{1}e_{0}^{3}-4e_{1}^{2}e_{0}^{2}$ &  & $e_{1}e_{0}e_{z}e_{0}$ & $-3e_{1}e_{0}^{3}-2e_{1}e_{0}e_{1}e_{0}-6e_{1}^{2}e_{0}^{2}$\tabularnewline
\cline{2-3} \cline{5-6} 
 & $e_{z}e_{1}e_{z}e_{0}$ & $-2e_{1}e_{0}e_{1}e_{0}-6e_{1}^{2}e_{0}^{2}-3e_{1}^{3}e_{0}$ &  & $e_{1}e_{0}^{2}e_{z}$ & $e_{1}e_{0}^{3}+e_{1}e_{0}e_{1}e_{0}+e_{1}^{2}e_{0}^{2}$\tabularnewline
\cline{2-3} \cline{5-6} 
 & $e_{z}e_{1}e_{0}e_{z}$ & $4e_{1}^{2}e_{0}^{2}+e_{1}^{3}e_{0}$ &  & $e_{1}e_{0}e_{1}e_{z}$ & $e_{1}e_{0}^{3}+e_{1}e_{0}e_{1}e_{0}+e_{1}^{2}e_{0}^{2}$\tabularnewline
\cline{2-3} \cline{5-6} 
 & $e_{z}e_{1}e_{0}^{2}$ & $-e_{1}e_{0}^{3}-e_{1}e_{0}e_{1}e_{0}-e_{1}^{2}e_{0}^{2}$ &  & $e_{1}^{2}e_{z}e_{0}$ & $-3e_{1}e_{0}^{3}-2e_{1}e_{0}e_{1}e_{0}-6e_{1}^{2}e_{0}^{2}$\tabularnewline
\cline{2-3} \cline{5-6} 
 & $e_{z}e_{1}^{2}e_{0}$ & $-e_{1}e_{0}^{3}-2e_{1}e_{0}e_{1}e_{0}-6e_{1}^{2}e_{0}^{2}-2e_{1}^{3}e_{0}$ &  & $e_{1}^{2}e_{0}e_{z}$ & $e_{1}e_{0}^{3}-e_{1}e_{0}e_{1}e_{0}-e_{1}^{2}e_{0}^{2}-2e_{1}^{3}e_{0}$\tabularnewline
\hline 
\end{tabular}

\caption{The table of $\lambda(w-\varphi_{\shuffle}(w))$ for a monomial $w$.}

\end{table}

\end{document}